\documentclass[10pt, reqno]{amsart}

\usepackage[pdftex]{graphicx}

\def\BibTeX{{\rm B\kern-.05em{\sc i\kern-.025em b}\kern-.08em
    T\kern-.1667em\lower.7ex\hbox{E}\kern-.125emX}}

\newtheorem{thm}{Theorem}[section]

\newtheorem{lem}[thm]{Lemma}
\newtheorem{prop}[thm]{Proposition}
\newtheorem{cond}[thm]{Condition}

\theoremstyle{definition}

\theoremstyle{remark}
\newtheorem{rem}{Remark}[section]
\newtheorem{exmp}{Example}[section]

\numberwithin{equation}{section}

    \newcommand{\EE}{\mathbb{E}}
    \newcommand{\ZZ}{\mathbb{Z}}

    \newcommand{\Exp}{\operatorname{E}}
    \newcommand{\E}{\Exp}

    \renewcommand{\Pr}{\operatorname{P}}

    \newcommand{\eqd}{\stackrel{d}{=}}
    \newcommand{\dto}{\xrightarrow{d}}
    \newcommand{\wto}{\xrightarrow{w}}
    \newcommand{\vto}{\xrightarrow{v}}
    \newcommand{\fidi}{\xrightarrow{\text{fidi}}}

    \newcommand{\eind}{\stackrel{d}{=}}

\begin{document}

\title[Weak convergence of partial maxima processes] 
{Weak convergence of partial maxima processes in the $M_{1}$ topology}

%
\author{Danijel Krizmani\'{c}}

\address{Danijel Krizmani\'{c}\\ Department of Mathematics\\
        University of Rijeka\\
        Radmile Matej\v{c}i\'{c} 2, 51000 Rijeka\\
        Croatia}
\email{dkrizmanic@math.uniri.hr}




\begin{abstract}
It is known that for a sequence of independent and identically
distributed random variables $(X_{n})$ the regular variation
condition is equivalent to weak convergence of partial maxima
$M_{n}= \max\{X_{1}, \ldots, X_{n}\}$, appropriately scaled. A
functional version of this is known to be true as well, the limit
process being an extremal process, and the convergence takes place
in the space of c\`{a}dl\`{a}g functions endowed with the Skorohod
$J_{1}$ topology. We first show that weak convergence of partial
maxima $M_{n}$ holds also for a class of weakly dependent
sequences under the joint regular variation condition. Then using
this result we obtain a corresponding functional version for the
processes of partial maxima $M_{n}(t) = \bigvee_{i=1}^{\lfloor nt
\rfloor}X_{i},\,t \in [0,1]$, but with respect to the Skorohod $M_{1}$
topology, which is weaker than the more usual $J_{1}$ topology. We
also show that the $M_{1}$ convergence generally can not be replaced by the $J_{1}$ convergence. Applications of our main results to moving
maxima, squared GARCH and ARMAX processes are also given.
\end{abstract}

\maketitle

\section{Introduction}
\label{intro}

Let $(X_{i})$ be a strictly stationary sequence of nonnegative
random variables and denote by $M_{n} = \max \{X_{i} :
i=1,\ldots,n\}$, $n \geq 1$, its accompanying sequence of partial
maxima.
 It is
well known that in the i.i.d. case weak convergence of $M_{n}$ is
equivalent to the regular variation property of $X_{1}$. More precisely, let
$(a_{n})$ be a sequence of positive real numbers
such that
 \begin{equation}\label{eq:an}
 n \Pr(X_{1}>a_{n}) \to 1 \qquad \textrm{as} \ n \to \infty,
 \end{equation}
   and $\mu$ a measure of the form
  $$ \mu(dx) = \alpha x^{-\alpha -1}1_{(0,\infty)}(x)\,dx$$
  for some $\alpha >0$. Then
\begin{equation}\label{eq:regvar}
 n \Pr \Big( \frac{X_{1}}
{a_{n}} \in \cdot \Big) \vto \mu (\,\cdot\,)
\end{equation}
  is equivalent to
 \begin{equation*}
\frac{M_{n}}{a_{n}} \dto Y_{0},
 \end{equation*}
 where $Y_{0}$ is a random variable with Fr\'{e}chet distribution
 $$ \Pr(Y_{0} \leq x) = e^{-\mu((x,\infty))} = e^{-x^{-\alpha}}, \qquad x \geq 0$$
 (see for example Resnick~\cite{Resnick07}, Proposition 7.1).
 The arrow $" \vto"$ in (\ref{eq:regvar}) denotes vague convergence of measures on
 $\mathbb{E} = (0, \infty]$. The regular variation
 property (\ref{eq:regvar}) is equivalent to
 $$ \Pr(X_{1} > x) = x^{-\alpha} L(x), \qquad x >0,$$
 where $L(\,\cdot\,)$ is a slowly varying function at $\infty$.

 In the i.i.d. case the regular variation property  (\ref{eq:regvar}) is also
 equivalent to the functional convergence of stochastic
 processes of partial maxima of $(X_{n})$, i.e.
 \begin{equation}
   M_{n}(\,\cdot\,) = \bigvee_{i=1}^{\lfloor n \cdot
   \rfloor}\frac{X_{i}}{a_{n}} \dto Y_{0}(\,\cdot\,)
 \end{equation}
in $D[0,1]$, the space of real-valued c\`{a}dl\`{a}g functions on
$[0,1]$, with the Skorohod $J_{1}$ topology, where $Y_{0}(\,\cdot\,)$
is an extremal process with exponent measure $\mu$, therefore with
marginal distributions
$$\Pr (Y_{0}(t) \leq x) = e^{-t
\mu((x,\infty))} = e^{-tx^{-\alpha}}, \quad x \geq 0,\,t \in
[0,1].$$
This result was proved by Lamperti~\cite{La64} (see also Resnick~\cite{Resnick07}, Proposition 7.2). For convenience we can put  $M_{n}(t)= X_{1}/a_{n}$ (or $M_{n}(t)=0$) for $t \in [0, 1/n)$. Weissman~\cite{We75} studied generalizations of
 this result to random variables which need not be identically distributed.
As for the dependent case, Adler~\cite{Ad78} obtained $J_{1}$ extremal functional convergence with the weak dependence condition similar to "asymptotical independence" condition introduced by Leadbetter~\cite{Le74} (see also Leadbetter~\cite{Le76}). $J_{1}$ functional convergence of sample extremal processes of moving averages was obtained by Davis and Resnick~\cite{DaRe85} with the noise having regularly varying tail probabilities, and by Jordanova~\cite{Jo09} with the noise in the Weibull max-domain of attraction. In the recent years different functional limit theorems for extremal processes subordinated to random time were obtained, see for instance Silvestrov and Teugels~\cite{SiTe98}; Meerschaert and Stoev~\cite{MeSt09}.

In this paper, under the properties of weak dependence and joint
regular variation with index $\alpha \in
(0,\infty)$ for the sequence $(X_{n})$, we investigate the asymptotic distributional behavior of
extremes $M_{n}$ and processes $M_{n}(\,\cdot\,)$. Since we study
extremes of random processes, nonnegativity of random variables
$X_{n}$ in reality is not a restrictive assumption. First, we introduce the essential
ingredients about regular variation and weak dependence in Section
\ref{S:Preliminaries}. In Section \ref{S:Weakconv} we prove the so
called timeless result on weak convergence of scaled extremes
$M_{n}$, based on a point process convergence obtained by Davis and Hsing~\cite{DH95}. Using this result and a limit theorem derived by Basrak et al.~\cite{BKS} for a certain time-space point processes,
in Section \ref{S:FLT} we prove a functional limit theorem for
processes of partial maxima $M_{n}(\,\cdot\,)$ in the space $D[0,1]$
endowed with the Skorohod $M_{1}$ topology, which is weaker than the frequently used
$J_{1}$ topology. The used methods are partly based on the work of Basrak et al.~\cite{BKS} for partial sums. Finally, in Section
\ref{S:Examples} we discuss several examples of stationary
sequences covered by our functional limit theorem, and show that
the $M_{1}$ convergence, in general, can not be replaced by the $J_{1}$ convergence.

\section{Preliminaries on regular variation, point processes and weak dependence}
\label{S:Preliminaries}

In this section we introduce the basic notions and results on regular variation and point processes that will be required for the results in the following sections.

Multivariate regular variation or regular variation on $\mathbb{R}_{+}^{d} = [0, \infty)^{d}$ for random vectors is typically formulated in terms of vague convergence on $\EE^{d}= [0, \infty]^{d} \setminus \{ \textbf{0} \}$. The topology on
$\EE^{d}$ is chosen so that a set $B \subseteq \EE^{d}$
has compact closure if and only if it is bounded away from zero,
that is, if there exists $u > 0$ such that $B \subseteq \EE^{d}_u = \{ \textbf{x}
\in \EE^{d} : \|\textbf{x}\| >u \}$. Here $\| \cdot \|$ denotes the max-norm on $\mathbb{R}_{+}^{d}$, i.e.\
$\displaystyle \|\textbf{x}\|=\max \{ x_{i} : i=1, \ldots , d\}$ where
$\textbf{x}=(x_{1}, \ldots, x_{d}) \in \mathbb{R}_{+}^{d}$.

The vector $\boldsymbol{\xi}$ with values in $\mathbb{R}_{+}^{d}$ is (multivariate) regularly varying with index $\alpha >0$ if there exists a random vector $\boldsymbol{\Theta}$
on the unit sphere $\mathbb{S}_{+}^{d-1} = \{ \textbf{x} \in \mathbb{R}_{+}^{d} :
\|\textbf{x}\|=1 \}$ in $\mathbb{R}_{+}^{d}$, such that for every $u \in (0,\infty)$
 \begin{equation}\label{e:regvar1}
   \frac{\Pr(\|\boldsymbol{\xi}\| > ux,\,\boldsymbol{\xi} / \| \boldsymbol{\xi} \| \in \cdot \, )}{\Pr(\| \boldsymbol{\xi} \| >x)}
    \wto u^{-\alpha} \Pr( \boldsymbol{\Theta} \in \cdot \,)
 \end{equation}
as $x \to \infty$, where the arrow "$\wto$" denotes weak convergence of finite measures.
Equivalently, $\boldsymbol{\xi}$ is regularly varying if there exists a sequence of positive real numbers $(a_{n})$ tending to infinity and a non-null Radon measure $\mu$ on $\mathcal{B}(\EE^{d})$ such that $\mu(\EE^{d} \setminus \mathbb{R}_{+}^{d})=0$ and
$$ n \Pr ( a_{n}^{-1} \boldsymbol{\xi} \in \cdot\,) \vto \mu (\,\cdot\,).$$

We say that a strictly stationary $\mathbb{R}_{+}$--valued process $(\xi_{n})$ is \emph{jointly regularly varying} with index
$\alpha \in (0,\infty)$ if for any nonnegative integer $k$ the
$k$-dimensional random vector $\boldsymbol{\xi} = (\xi_{1}, \ldots ,
\xi_{k})$ is multivariate regularly varying with index $\alpha$.

Theorem 2.1 in Basrak and Segers~\cite{BaSe} provides a convenient
characterization of joint regular variation: it is necessary and
sufficient that there exists a process $(Y_n)_{n \in \mathbb{Z}}$
with $\Pr(Y_0 > y) = y^{-\alpha}$ for $y \geq 1$ such that as
$x \to \infty$,
\begin{equation}\label{e:tailprocess}
  \bigl( (x^{-1}\xi_n)_{n \in \ZZ} \, \big| \, \xi_0 > x \bigr)
  \fidi (Y_n)_{n \in \ZZ},
\end{equation}
where "$\fidi$" denotes convergence of finite-dimensional
distributions. The process $(Y_{n})$ is called
the \emph{tail process} of $(\xi_{n})$.

Let $(X_{i})$ be a strictly stationary sequence of nonnegative
random variables and assume it is jointly regularly varying with index $\alpha \in (0, \infty)$.
The property of joint regular variation is a corner stone in obtaining the weak convergence of point processes $N_{n}$ given by
$$N_{n} = \sum_{i=1}^{n}\delta_{X_{i}/a_{n}}, \qquad \ n \in \mathbb{N},$$
with $a_{n}$ as in (\ref{eq:an}). These point processes play a fundamental role in obtaining the limit theorem for scaled extremes $M_{n}$. On the other part the time-space point processes
\begin{equation}
\label{E:ppspacetime}
  N^{*}_{n} = \sum_{i=1}^{n} \delta_{(i / n,\,X_{i} / a_{n})}, \qquad n \in \mathbb{N},
\end{equation}
are connected with the functional limit theorem for processes of partial maxima $M_{n}(\,\cdot\,)$.

To control the dependence in the sequence $(X_n)$ we first have to assume that clusters of large values of $X_{n}$ do not last for too long.
\begin{cond}
\label{c:anticluster}
There exists a sequence of positive integers $(r_{n})$ such that $r_{n} \to \infty $ and $r_{n} / n \to 0$ as $n \to \infty$ and such that for every $u > 0$,
\begin{equation}
\label{e:anticluster}
  \lim_{m \to \infty} \limsup_{n \to \infty}
  \Pr \biggl( \max_{m \leq |i| \leq r_{n}} X_{i} > ua_{n}\,\bigg|\,X_{0}>ua_{n} \biggr) = 0.
\end{equation}
\end{cond}
Under the finite-cluster Condition~\ref{c:anticluster} the following value
\begin{equation}\label{e:theta}
   \theta  = \lim_{r \to \infty} \lim_{x \to \infty} \Pr \bigl(M_{r} \leq x \, \big| \, X_{0}>x \bigr)
\end{equation}
is strictly positive, and it is equal to the extremal index of the sequence $(X_{n})$ (see Basrak and Segers~\cite{BaSe}). The extremal index can be interpreted as the reciprocal mean cluster size of large exceedances (cf. Hsing et al.~\cite{HHL88}). Clustering of extreme values occurs when $\theta < 1$.

The weak dependence condition appropriate for our considerations is the mixing condition called $\mathcal{A}'(a_n)$ which is slightly stronger than the condition~$\mathcal{A}(a_n)$ introduced in Davis and Hsing~\cite{DH95}. Condition $\mathcal{A}'(a_n)$ is implied by the strong mixing property, see Krizmani'{c}~\cite{Kr10}. Recall that a sequence of random variables $(\xi_{n})$ is strongly mixing if $\alpha (n) \to 0$ as $n \to \infty$, where
$$\alpha (n) = \sup \{|\Pr (A \cap B) - \Pr(A) \Pr(B)| : A \in \mathcal{F}_{1}^{j}, B \in \mathcal{F}_{j+n}^{\infty}, j=1,2, \ldots \}$$
and $\mathcal{F}_{k}^{l} = \sigma( \{ \xi_{i} : k \leq i \leq l \} )$ for $1 \leq k \leq l \leq \infty$.

\begin{cond}[$\mathcal{A}'(a_{n})$]
\label{c:mixcond}
There exists a sequence of positive integers $(r_{n})$ such that $r_{n} \to \infty $ and $r_{n} / n \to 0$ as $n \to \infty$ and such that for every nonnegative continuous function $f$ on $[0,1] \times \EE$ with compact support, denoting $k_{n} = \lfloor n / r_{n} \rfloor$, as $n \to \infty$,
\begin{equation}\label{e:mixcon}
 \Exp \biggl[ \exp \biggl\{ - \sum_{i=1}^{n} f \biggl(\frac{i}{n}, \frac{X_{i}}{a_{n}}
 \biggr) \biggr\} \biggr]
 - \prod_{k=1}^{k_{n}} \Exp \biggl[ \exp \biggl\{ - \sum_{i=1}^{r_{n}} f \biggl(\frac{kr_{n}}{n}, \frac{X_{i}}{a_{n}} \biggr) \biggr\} \biggr] \to 0.
\end{equation}
\end{cond}

Under joint regular variation and Conditions \ref{c:anticluster} and \ref{c:mixcond}, by Theorem 2.7 in Davis and Hsing~\cite{DH95} we obtain that the point processes $N_{n}$, $n \in \mathbb{N}$, converge in distribution to some $N$, which by Theorem 2.3 and Corollary 2.4 in Davis and Hsing~\cite{DH95} has the following cluster representation
\begin{equation}\label{eq:Nnconv}
 N \eind \sum_{i} \sum_{j} \delta_{P_{i}Q_{ij}},
 \end{equation}
where $\sum_{i=1}^{\infty}\delta_{P_{i}}$ is a Poisson process
with intensity measure $\nu$ given by $\nu(dy) = \theta \alpha
y^{-\alpha-1}1_{(0,\infty)}(y)\,dy$, and $\sum_{j= 1}^{\infty}\delta_{Q_{ij}}$, $i \geq 1$,
are i.i.d. point processes on $[0,1]$ whose points satisfy $\sup_{j}Q_{ij}=1$, and all point processes are mutually
independent. The distribution of the point processes $\sum_{j= 1}^{\infty}\delta_{Q_{ij}}$ is described in Davis and Hsing~\cite{DH95}.

Conditions \ref{c:anticluster} and \ref{c:mixcond}, by Theorem 2.3 in Basrak et al.~\cite{BKS},
also imply convergence in distribution of the point processes $N_{n}^{*}$ on the set
$[0,1] \times \EE_{u}$ for every $u \in (0,\infty)$, where $\EE_{u}= (u,\infty]$.
More precisely, under the above conditions, for every $u \in (0,\infty)$, as $n \to \infty$,
\begin{equation}\label{e:Nn*conv}
N_{n}^{*} \bigg|_{[0, 1] \times \EE_u} \dto N^{(u)}
    = \sum_i \sum_j \delta_{(T^{(u)}_i, u Z_{ij})} \bigg|_{[0, 1] \times \EE_u}\,
\end{equation}
in $[0, 1] \times \EE_u$, where $\sum_i \delta_{T^{(u)}_i}$ is a
homogeneous Poisson process on $[0, 1]$ with intensity $\theta
u^{-\alpha}$, and $(\sum_j \delta_{Z_{ij}})_i$ is an i.i.d.\
sequence of point processes in $\EE$, independent of $\sum_i
\delta_{T^{(u)}_i}$, and with common distribution equal to the
distribution of
$$\biggl( \sum_{n \in \mathbb{Z}} \delta_{Y_n} \, \bigg| \, \sup_{i \leq -1} Y_i \leq 1 \biggr),$$
where $(Y_{n})$ is the tail process of the sequence $(X_{n})$.

For a detailed discussion on joint regular variation and dependence conditions \ref{c:anticluster} and \ref{c:mixcond} we refer to Basrak et al.~\cite{BKS}, Section 3.4.

\section{Weak convergence of partial maxima $M_{n}$}
\label{S:Weakconv}

In this section we establish convergence of the partial maxima $M_{n}$ for a class of weakly dependent sequences. Precisely, let $(X_{n})$ be a strictly stationary sequence of nonnegative random variables, jointly regularly varying with index $\alpha \in (0, \infty)$ and assume Conditions \ref{c:anticluster} and \ref{c:mixcond} hold. Then by (\ref{eq:Nnconv}) it holds that, as $n \to \infty$,
$$ N_{n} = \sum_{i=1}^{n}\delta_{X_{i}/a_{n}} \dto N = \sum_{i} \sum_{j} \delta_{P_{i}Q_{ij}},$$
where $(a_{n})$ is chosen as in (\ref{eq:an}). Denote by
$\mathbf{M}_{p}(\EE)$ the space of Radon point measures on $\EE$
equipped with the vague topology. Recall $M_{n} =
\bigvee_{i=1}^{n} X_{i}$.

\begin{thm}\label{t:notimeconv}
Let $(X_{n})$ be a strictly stationary sequence of nonnegative random variables, jointly regularly varying with index $\alpha\in(0, \infty)$. Suppose that Conditions~\ref{c:anticluster} and \ref{c:mixcond} hold. Then, as $n \to \infty$,
$$ \frac{M_{n}}{a_{n}} \dto M,$$
where the limit $M$ is a Fr\'{e}chet random variable with
$$ \Pr(M \leq x) = e^{-\theta x^{-\alpha}}, \qquad x >0.$$
\end{thm}

\begin{proof}
Define $M = \bigvee_{i=1}^{\infty} \bigvee_{j=1}^{\infty}P_{i}Q_{ij}$, and let $\epsilon >0$ be arbitrary. The mapping $T_{\epsilon} \colon
\mathbf{M}_{p}(\EE) \to \mathbb{R}$ defined by
$$ T_{\epsilon} \Big( \sum_{i=1}^{\infty}\delta_{x_{i}} \Big) =
\bigvee_{i=1}^{\infty}x_{i} 1_{\{ x_{i} \in [\epsilon, \infty)
\}}$$ is continuous on the set $\Lambda_{\epsilon} = \{ \eta \in
\mathbf{M}_{p}(\EE) : \eta (\{ \epsilon \})=0 \}$ (cf. Resnick~\cite{Resnick87}, page 214). Since $N$ has no fixed atoms (see Lemma 2.1 in
Davis and Hsing~\cite{DH95}), i.e. $\Pr(N \in \Lambda_{\epsilon})=1$, using the
continuous mapping theorem we obtain
\begin{equation}\label{e:convTeps}
M_{n}[\epsilon, \infty)= T_{\epsilon}(N_{n}) \dto T_{\epsilon}(N) = M[\epsilon, \infty) \qquad \textrm{as} \ n \to \infty,
\end{equation}
with the notation
$$ M_{n}B = a_{n}^{-1} \bigvee_{i=1}^{n}X_{i} 1_{\{ a_{n}^{-1}X_{i} \in B \}},$$
and
$$ M B = \bigvee_{i=1}^{\infty} \bigvee_{j=1}^{\infty}P_{i}Q_{ij} 1_{\{ P_{i}Q_{ij} \in B \}}$$
for any Borel set $B$ in $\mathbb{R}$.
Obviously
\begin{equation}\label{e:convas}
M[\epsilon, \infty) \to M(0,\infty) = M
\end{equation}
almost surely as $\epsilon \to 0$. If we show that
\begin{equation}\label{eq:slutskycond}
 \lim_{\epsilon \to 0} \limsup_{n \to \infty} \Pr (|M_{n}[\epsilon,\infty) - M_{n}(0,\infty)| > \delta)=0
 \end{equation}
for any $\delta >0$, then by Theorem 3.5 in Resnick~\cite{Resnick07} we will have $M_{n}(0,\infty) \dto M(0,\infty)$, i.e. $a_{n}^{-1} M_{n} \dto M$ as $n \to \infty$.

Since for arbitrary real numbers $x_{1}, \ldots, x_{n}, y_{1},
\ldots, y_{n}$ the following inequality
 \begin{equation}\label{e:maxineq}
 \Big| \bigvee_{i=1}^{n}x_{i} - \bigvee_{i=1}^{n}y_{i} \Big| \leq
\bigvee_{i=1}^{n}|x_{i}-y_{i}|
 \end{equation}
  holds, note that
\begin{equation*}
  |M_{n}[\epsilon,\infty) - M_{n}(0,\infty)|  \leq M_{n}(0,\epsilon).
\end{equation*}
Take an arbitrary $s > \alpha$. Then using stationarity and Markov's inequality we get the bound
 \begin{eqnarray}\label{eq:slutsky1}
   \nonumber \Pr ( M_{n}(0,\epsilon) > \delta) & \leq & n \Pr \Big( \frac{X_{1}}{a_{n}} 1_{\{X_{1} < \epsilon a_{n}\}} > \delta \Big)
    \leq \frac{n}{\delta^{s} a_{n}^{s}} \E ( X_{1}^{s} 1_{\{ X_{1} < \epsilon a_{n} \}})\\[0.4em]
     & \hspace*{-6em} = &  \hspace*{-3em} \frac{\epsilon^{s}}{\delta^{s}} \cdot n \Pr(X_{1} >  a_{n}) \cdot \frac{\Pr(X_{1} > \epsilon a_{n})}{\Pr(X_{1} > a_{n})} \cdot \frac{\E (X_{1}^{s} 1_{\{ X_{1} < \epsilon a_{n} \}})}{\epsilon^{s} a_{n}^{s} \Pr(X_{1}>\epsilon a_{n})}.
 \end{eqnarray}
Since the distribution of $X_{1}$ is regularly varying with index
$\alpha$, it follows immediately that
$$ \frac{\Pr(X_{1}> \epsilon a_{n})}{\Pr(X_{1}>a_{n})} \to
    \epsilon^{-\alpha}$$
 as $n \to \infty$. By Karamata's theorem
 $$ \lim_{n \to \infty} \frac{\E(X_{1}^{s} \, 1_{ \{ X_{1} < \epsilon a_{n} \}
            })}{\epsilon^{s} a_{n}^{s}\Pr(X_{1}> \epsilon a_{n})} =
            \frac{\alpha}{s-\alpha}.$$
  Thus from (\ref{eq:slutsky1}), taking into account
 relation (\ref{eq:an}), we get
 $$ \limsup_{n \to \infty} \Pr(M_{n}(0,\epsilon)>\delta) \leq \delta^{-s}
 \frac{\alpha}{s-\alpha}\epsilon^{s-\alpha}.$$
 Letting $\epsilon \to 0$, since $s-\alpha >0$, we finally obtain
 $$ \lim_{\epsilon \to 0} \limsup_{n \to \infty} \Pr(M_{n}(0,\epsilon)>\delta) = 0,$$
 and relation (\ref{eq:slutskycond}) holds. Therefore $a_{n}^{-1} M_{n} \dto M$ as $n \to \infty$.

 From the representation in (\ref{eq:Nnconv}) and the fact that $\sup_{j}Q_{ij}=1$ we obtain the distribution of the limit $M$,
 $$ \Pr(M \leq x) = \Pr \Big( \bigvee_{i=1}^{\infty}P_{i} \leq x \Big) = \Pr \Big( \sum_{i}\delta_{P_{i}}(x,\infty) = 0 \Big) = e^{-\nu(x,\infty)} = e^{-\theta x^{-\alpha}}$$
 for $x>0$.
\end{proof}

\section{Functional convergence of partial maxima processes $M_{n}(\,\cdot\,)$}
\label{S:FLT}

In this section we show the convergence of the partial maxima
processes $M_n(\,\cdot\,)$ to an extremal process in the space $D[0, 1]$
equipped with the Skorohod $M_1$ topology. Similar to the case of
partial sum processes in Basrak et al.~\cite{BKS} we first represent
the partial maxima process $M_n(\,\cdot\,)$ as the image of the
time-space point process $N_n^{*}$ under a certain maximum
functional. Then, using certain continuity properties of this
functional, the continuous mapping theorem and the standard
"finite dimensional convergence plus tightness" procedure we
transfer the weak convergence of $N_n^{*}$ in (\ref{e:Nn*conv}) to
weak convergence of $M_n(\,\cdot\,)$.

Recall the definition of the $M_{1}$ topology. For $x \in D[0,1]$
the \emph{completed graph} of $x$ is the set
\[
  \Gamma_{x}
  = \{ (t,z) \in [0,1] \times \mathbb{R} : z= \lambda x(t-) + (1-\lambda)x(t) \ \text{for some}\ \lambda \in [0,1] \},
\]
where $x(t-)$ is the left limit of $x$ at $t$. Besides the points
of the graph $ \{ (t,x(t)) : t \in [0,1] \}$, the completed graph
of $x$ also contains the vertical line segments joining $(t,x(t))$
and $(t,x(t-))$ for all discontinuity points $t$ of $x$. We define
an \emph{order} on the graph $\Gamma_{x}$ by saying that
$(t_{1},z_{1}) \leq (t_{2},z_{2})$ if either (i) $t_{1} < t_{2}$ or
(ii) $t_{1} = t_{2}$ and $|x(t_{1}-) - z_{1}| \leq |x(t_{2}-) -
z_{2}|$. A \emph{parametric representation} of the completed graph
$\Gamma_{x}$ is a continuous nondecreasing function $(r,u)$
mapping $[0,1]$ onto $\Gamma_{x}$, with $r$ being the time
component and $u$ being the spatial component.  Let $\Pi(x)$
denote the set of parametric representations of the graph
$\Gamma_{x}$. For $x_{1},x_{2} \in D[0,1]$ define
\[
  d_{M_{1}}(x_{1},x_{2})
  = \inf \{ \|r_{1}-r_{2}\|_{[0,1]} \vee \|u_{1}-u_{2}\|_{[0,1]} : (r_{i},u_{i}) \in \Pi(x_{i}), i=1,2 \},
\]
where $\|x\|_{[0,1]} = \sup \{ |x(t)| : t \in [0,1] \}$ and $a
\vee b = \max\{a,b\}$. $d_{M_{1}}$ is a metric on $D[0,1]$, and
the induced topology is called the Skorohod $M_{1}$ topology. This
topology is weaker than the more frequently used Skorohod $J_{1}$
topology. For more discussion of the $M_{1}$ topology we refer to Whitt~\cite{Whitt02}, sections 12.3-12.5.

Fix $0 < v < u < \infty$. Define the maximum functional
$$
  \phi^{(u)} \colon \mathbf{M}_{p}([0,1] \times \EE_{v}) \to
  D[0,1]
$$
 by
 $$ \phi^{(u)} \Big( \sum_{i}\delta_{(t_{i},\,x_{i})} \Big) (t)
  =  \bigvee_{t_{i} \leq t} x_{i} \,1_{\{u < x_i < \infty\}}, \qquad t \in [0,
  1],$$
  where the supremum of an empty set may be taken, for
  convenience, to be $\min\{0, \bigwedge_{i}x_{i}\}$.
Note that $\phi^{(u)}$ is well defined because $[0,1] \times
\EE_{u}$ is a relatively compact subset of $[0,1] \times \EE_{v}$.
The space $\mathbf{M}_p([0,1] \times \EE_{v})$ of Radon point
measures on $[0,1] \times \EE_{v}$ is equipped with the vague
topology and $D[0, 1]$ is equipped with the $M_1$ topology. Let
$$\Lambda = \{ \eta \in \mathbf{M}_{p}([0,1] \times \EE_{v}) : \eta
(\{0,1\} \times \EE_{u}) = \eta ([0,1] \times \{u, \infty \})=0\}.$$
Observe that elements of $\Lambda$ are Radon point
measures that have no atoms on the border of $[0,1] \times
\EE_{u}$. Then the point process $N^{(v)}$ defined in
(\ref{e:Nn*conv}) almost surely belongs to the set $\Lambda$, see
Lemma 3.1 in Basrak et al.~\cite{BKS}. Now we will show that
$\phi^{(u)}$ is continuous on the set $\Lambda$.

\begin{lem}\label{l:contfunct}
The maximum functional $\phi^{(u)} \colon \mathbf{M}_{p}([0,1]
\times \EE_{v}) \to D[0,1]$ is continuous on the set $\Lambda$,
when $D[0,1]$ is endowed with the Skorohod $M_{1}$ topology.
\end{lem}

\begin{proof}
Take an arbitrary $\eta \in \Lambda$ and suppose that $\eta_{n}
\vto \eta$ in $\mathbf{M}_{p}([0,1] \times \EE_{v})$. We will show
that $\phi^{(u)}(\eta_{n}) \to \phi^{(u)}(\eta)$ in $D[0,1]$
according to the $M_{1}$ topology.
Since the set $[0,1] \times \EE_{u}$ is relatively compact in
$[0,1] \times \EE_{v}$, there exists a nonnegative integer
$k=k(\eta)$ such that
$$ \eta ([0,1] \times \EE_{u}) = k < \infty.$$
By assumption, $\eta$ does not have any atoms on the border of the
set $[0,1] \times \EE_{u}$. Hence, by Lemma 7.1 in Resnick~\cite{Resnick07}, there exists a positive integer $n_{0}$ such
that for all $n \geq n_{0}$ it holds that
$$ \eta_{n} ([0,1] \times \EE_{u})=k.$$
If $k=0$ there is nothing to prove, so assume $k \geq 1$, and let
$(t_{i},x_{i})$ for $i=1,\ldots,k$ be the atoms of $\eta$ in
$[0,1] \times \EE_{u}$. By the same lemma, the $k$ atoms
$(t_{i}^{(n)}, x_{i}^{(n)})$ of $\eta_{n}$ in $[0,1] \times
\EE_{u}$ (for $n \geq n_{0}$) can be labelled in such a way that
for every $i \in \{1,\ldots,k\}$ we have
$$ (t_{i}^{(n)}, x_{i}^{(n)}) \to (t_{i},x_{i}) \qquad \textrm{as}
\ n \to \infty.$$ In particular, for any $\delta >0$ we can find a
positive integer $n_{\delta} \geq n_{0}$ such that for all $n \geq
n_{\delta}$,
\begin{equation}\label{e:etaconv}
 |t_{i}^{(n)} - t_{i}| < \delta \quad \textrm{and} \quad
 |x_{i}^{(n)}- x_{i}| < \delta \qquad \textrm{for} \ i=1,\ldots,k.
\end{equation}
Let the sequence
$$ 0 < \tau_{1} < \tau_{2} < \ldots < \tau_{p} < 1$$
be such that the sets $\{\tau_{1}, \ldots, \tau_{p}\}$ and
$\{t_{1}, \ldots, t_{k}\}$ coincide. Since $\eta$ can have several
atoms with the same time coordinate, it always  holds that $p \leq k$.
Put $\tau_{0}=0$, $\tau_{p+1}=1$, and take
$$ 0 < r < \frac{1}{2} \min_{0 \leq i \leq p}|\tau_{i+1} -
\tau_{i}|.$$
 For any $t \in [0,1] \setminus \{\tau_{1}, \ldots, \tau_{p}\}$ we
 can find $\delta \in (0,u)$ such that
 $$ \delta < r \quad \textrm{and} \quad \delta < \min_{1 \leq i \leq
 p}|t - \tau_{i}|.$$
 Then relation (\ref{e:etaconv}), for $n \geq n_{\delta}$, implies
 that $t_{i}^{(n)} \leq t$ is equivalent to $t_{i} \leq t$, and we
 obtain
 $$ |\phi^{(u)}(\eta_{n})(t) - \phi^{(u)}(\eta)(t)| = \bigg| \bigvee_{t_{i}^{(n)} \leq
 t}x_{i}^{(n)} - \bigvee_{t_{i}\leq t}x_{i} \bigg| \leq \bigvee_{t_{i} \leq
 t}|x_{i}^{(n)}-x_{i}| < \delta.$$
 Therefore
 $$\lim_{n \to \infty}|\phi^{(u)}(\eta_{n})(t) -
 \phi^{(u)}(\eta)(t)|< \delta,$$
 and if we let $\delta \to 0$, it follows that
 $\phi^{(u)}(\eta_{n})(t) \to \phi^{(u)}(\eta)(t)$ as $n \to
 \infty$.
Note that the functions $\phi^{(u)}(\eta)$ and
$\phi^{(u)}(\eta_{n})\,(n \geq n_{\delta})$ are monotone. Since,
by Corollary 12.5.1 in Whitt~\cite{Whitt02}, $M_{1}$ convergence
for monotone functions is equivalent to pointwise convergence in a
dense subset of points plus convergence at the endpoints, we
obtain that $d_{M_{1}}(\phi^{(u)}(\eta_{n}), \phi^{(u)}(\eta)) \to
0$ as $n \to \infty$, i.e. $\phi^{(u)}$ is continuous at $\eta$.
\end{proof}

The theorem below gives conditions under which the partial maxima
processes of a strictly stationary, jointly regularly varying
sequence of nonnegative random variables satisfies a functional
limit theorem with an extremal process as a limit. Extremal processes can be defined by Poisson processes in the following way. Let $\xi = \sum_{k}\delta_{(t_{k}, j_{k})}$ be a Poisson process on $(0,\infty) \times \mathbb{E}$ with mean measure $\lambda \times \nu$, where $\lambda$ is the Lebesgue measure. The extremal process $\widetilde{M}(\,\cdot\,)$ generated by $\xi$ is defined by
$$ \widetilde{M}(t) = \sup \{ j_{k} : t_{k} \leq t\}, \qquad t>0.$$
The distribution function of $\widetilde{M}(\,\cdot\,)$ is of the form
$$ \Pr( \widetilde{M}(t) \leq x) = e^{-t \nu(x,\infty)}$$
for $t>0$ (cf. Resnick~\cite{Resnick86}). The measure $\nu$ is called the exponent measure.

The convergence in the theorem takes place in the space $D[0,1]$ endowed with the Skorohod $M_{1}$ topology. In the proof of the theorem we will need a characterization of $M_{1}$ convergence for random processes which is due to Skorohod. Put
$$ M(x_{1},x_{2},x_{3}) = \left\{ \begin{array}{ll}
                                   0, & \ \ \textrm{if} \ x_{2} \in [x_{1}, x_{3}], \\
                                   \min\{ |x_{2}-x_{1}|, |x_{3}-x_{2}| \}, & \ \ \textrm{otherwise},
                                 \end{array}\right.$$
(note that $M(x_{1},x_{2},x_{3})$ is the distance from $x_{2}$ to $[x_{1},x_{3}]$) and introduce the $M_{1}$ oscillation $\omega_{\delta}(x)$ of a function $x \in D[0,1]$ by
$$ \omega_{\delta}(x) = \sup_{{\footnotesize \begin{array}{c}
                                t_{1} \leq t \leq t_{2} \\
                                0 \leq t_{2}-t_{1} \leq \delta
                              \end{array}}
} M(x(t_{1}), x(t), x(t_{2})),$$
for $\delta >0$. Then the following corollary of Theorems 3.2.1 and 3.2.2 in Skorohod~\cite{Sk56} holds.

\begin{prop}\label{pr:Skorohod}
Let $Z_{n}(\,\cdot\,)$ be processes in $D[0,1]$ whose finite dimensional distributions converge to those of a process $Z(\,\cdot\,)$ which is a.s. continuous at $t=0$ and $t=1$. Then $Z_{n}(\,\cdot\,)$ converges in distribution to $Z(\,\cdot\,)$ in $D[0,1]$ with respect to the Skorohod $M_{1}$ topology if and only if for every $\epsilon >0$,
\begin{equation}\label{e:oscM1}
 \lim_{\delta \to 0} \limsup_{n \to \infty} \Pr ( \omega_{\delta}(Z_{n}(\,\cdot\,)) > \epsilon ) =0.
 \end{equation}
\end{prop}

\begin{rem}\label{r:J1conv}
The statement of Proposition~\ref{pr:Skorohod} remains valid if the $M_{1}$ topology is replaced by the $J_{1}$ topology, and the $M_{1}$ oscillation $\omega_{\delta}(\,\cdot\,)$ is replaced by the $J_{1}$ oscillation $\omega_{\delta}'(\,\cdot\,)$ defined by
$$ \omega_{\delta}'(x) = \sup_{{\footnotesize \begin{array}{c}
                                t_{1} \leq t \leq t_{2} \\
                                0 \leq t_{2}-t_{1} \leq \delta
                              \end{array}}
} \min\{| x(t)-x(t_{1})|, |x(t_{2})-x(t)| \},$$
for $x \in D[0,1]$ and $\delta >0$ (see Skorohod~\cite{Sk56}).
\end{rem}

\begin{thm}\label{t:functconvergence}
Let $(X_{n})$ be a strictly stationary sequence of nonnegative
random variables, jointly regularly varying with index $\alpha \in
(0, \infty)$. Suppose that Conditions \ref{c:anticluster} and
\ref{c:mixcond} hold. Then the partial maxima stochastic process
$$  M_{n}(t) = \bigvee_{i=1}^{\lfloor nt
   \rfloor}\frac{X_{i}}{a_{n}}, \qquad t \in [0,1],$$
satisfies
$$ M_{n}(\,\cdot\,) \dto \widetilde{M}(\,\cdot\,), \qquad n \to \infty,$$
in $D[0,1]$ endowed with the $M_{1}$ topology, where
$\widetilde{M}(\,\cdot\,)$ is an extremal process with exponent
measure $\nu(x,\infty) = \theta x^{-\alpha}$, $x>0$, with $\theta$ given by $(\ref{e:theta})$.
\end{thm}

\begin{proof}
Using the techniques from the proof of Theorem 3.4 in Basrak et al.~\cite{BKS} we obtain that the point process
$$ \widehat{N}^{(u)} = \sum_{i} \delta_{(T_{i}^{(u)}, u \bigvee_{j} Z_{ij}1_{\{ Z_{ij} >1
\}})},$$ is a Poisson process with mean measure $\lambda \times \nu^{(u)}$, where
the measure $\nu^{(u)}$ is defined by
$$  \nu^{(u)}(x, \infty) =  u^{-\alpha} \,
\Pr \biggl( u \bigvee_{i \geq 0} Y_i \, 1_{\{Y_i > 1\}} > x, \,
\sup_{i \leq -1} Y_i \leq 1 \biggr), \qquad x >0,$$
with $(Y_{i})$ being the tail process of the sequence $(X_{i})$.

Consider now $0<u<v$ and
$$  \phi^{(u)} (N_{n}^{*}\,|\,_{[0,1] \times \EE_{u}}) (\,\cdot\,)
  = \phi^{(u)} (N_{n}^{*}\,|\,_{[0,1] \times \EE_{v}}) (\,\cdot\,)
  = \bigvee_{i/n \leq \, \cdot} \frac{X_{i}}{a_{n}} 1_{ \big\{ \frac{X_{i}}{a_{n}} > u
    \big\} },$$
which by Lemma~\ref{l:contfunct} and the continuous mapping theorem converges in distribution in $D[0,1]$ under the $M_{1}$ metric to
$$
\phi^{(u)}
(N^{(v)})(\,\cdot\,)
 =\phi^{(u)} (N^{(v)}\,|\,_{[0,1] \times \EE_{u}})(\,\cdot\,).
$$
Since by the definition of the process $N^{(u)}$ in
(\ref{e:Nn*conv}) it holds that
$ N^{(u)} \eind
N^{(v)} |_{[0, 1] \times \EE_u}$,
the last expression above is equal in distribution to
$$
\phi^{(u)} (N^{(u)})(\,\cdot\,)
 = \bigvee_{T_{i}^{(u)} \leq \, \cdot}
   \bigvee_{j}uZ_{ij}1_{ \{ Z_{ij} > 1 \} }.
$$
 But since
  $\phi^{(u)}(N^{(u)}) = \phi^{(u)} (\widehat{N}^{(u)})\,\eind\,\phi^{(u)} (\widetilde{N}^{(u)})$,
 where
 $$ \widetilde{N}^{(u)} = \sum_{i} \delta_{(T_{i},\,K_{i}^{(u)})}
 $$
 is a Poisson process (or Poisson random measure) with mean measure $\lambda \times \nu^{(u)}$,
 we obtain
 \begin{equation}\label{eq:convaboveu}
  M_{n}^{(u)}(\,\cdot\,) := \bigvee_{i = 1}^{\lfloor n \, \cdot \, \rfloor} \frac{X_{i}}{a_{n}} 1_{ \big\{ \frac{X_{i}}{a_{n}} > u
    \big\} } \dto M^{(u)}(\,\cdot\,) := \bigvee_{T_{i} \leq \, \cdot} K_{i}^{(u)} \quad \text{as} \ n \to \infty,
 \end{equation}
 in $D[0,1]$ under the $M_{1}$ metric.
Note that the limiting process $M^{(u)}(\,\cdot\,)$ is an extremal process with exponent measure $\nu^{(u)}$, since
$$ \Pr (M^{(u)}(t) \leq x) = \Pr(\widetilde{N}^{(u)}((0,t] \times (x,\infty))=0) = e^{-t\nu^{(u)}(x,\infty)}$$
for $t \in [0,1]$ and $x >0$.

Since the function $\pi \colon D[0,1] \to \mathbb{R}$ defined by $\pi(x)=x(1)$ is continuous (see Theorem 12.5.1 (iv) in Whitt~\cite{Whitt02}), from (\ref{eq:convaboveu}) using the continuous mapping theorem, we obtain
\begin{equation}\label{e:convMu1}
M_{n}^{(u)}(1) \dto M^{(u)}(1) \qquad \textrm{as} \ n \to \infty.
\end{equation}
If we now apply the notation from the proof of Theorem~\ref{t:notimeconv}, we see that $M_{n}^{(u)}(1) = M_{n}(u,\infty)$. Therefore comparing (\ref{e:convTeps})  and (\ref{e:convMu1}) we conclude that $M^{(u)}(1) \eqd M (u,\infty)$. Further, from (\ref{e:convas}) it follows that
$M^{(u)}(1) \dto M$ as $u \to 0$. Therefore taking into account the distribution of $M$ we conclude that
$ e^{-\nu^{(u)}(x,\infty)} \to e^{-\nu(x,\infty)}$
for all $x >0$ that are continuity points of the distribution of $M$, where $\nu(dy) = \theta \alpha
y^{-\alpha-1}1_{(0,\infty)}(y)\,dy$. Hence
\begin{equation}\label{e:convnuu}
\nu^{(u)}(x,\infty) \to \nu(x,\infty) \qquad \textrm{as} \ u \to 0,
\end{equation}
 for every continuity point $x$ of $\nu(\,\cdot, \infty)$.

Now we show that the finite dimensional distributions of $M^{(u)}(\,\cdot\,)$ converge, as $u$ tends to zero, to the finite dimensional distributions of an extremal process $\widetilde{M}(\,\cdot\,)$ generated by a Poisson process $\sum_{i}\delta_{(T_{i},K_{i})}$ with mean measure $\lambda \times \nu$, i.e. $\widetilde{M}(t) = \bigvee_{T_{i} \leq t} K_{i}$, $t \in [0,1]$. Since $M^{(u)}(\,\cdot\,)$ is an extremal process, its finite dimensional distributions are of the form
\begin{eqnarray*}
   \Pr(M^{(u)}(t_{1}) \leq x_{1}, \ldots, M^{(u)}(t_{k}) \leq x_{k}) & &  \\[0.3em]
   & \hspace*{-25em} = & \hspace*{-12em} e^{-t_{1}\nu^{(u)}(\bigwedge_{i=1}^{k}x_{i},\infty)} \cdot e^{-(t_{2}-t_{1}) \nu^{(u)}(\bigwedge_{i=2}^{k}x_{i},\infty)} \cdot \ldots \cdot e^{-(t_{k}-t_{k-1})\nu^{(u)}(x_{k},\infty)},
\end{eqnarray*}
for $0 \leq t_{1} < t_{2} < \ldots < t_{k} \leq 1$ and positive real numbers $x_{1},\ldots,x_{k}$ (see Resnick~\cite{Resnick86}, Section 2.3). Letting $u \to 0$ and using (\ref{e:convnuu}) we immediately obtain that the right hand side in the last equation above converges (in the continuity points $x_{1}, \ldots, x_{k}$ of $\nu(\,\cdot, \infty)$) to
$$ e^{-t_{1}(\bigwedge_{i=1}^{k}x_{i},\infty)} \cdot e^{-(t_{2}-t_{1}) \nu(\bigwedge_{i=2}^{k}x_{i},\infty)} \cdot \ldots \cdot e^{-(t_{k}-t_{k-1})\nu(x_{k},\infty)}.$$
But since this limit is in fact $\Pr(\widetilde{M}(t_{1}) \leq x_{1}, \ldots, \widetilde{M}(t_{k}) \leq x_{k})$, we conclude that the finite dimensional distributions of $M^{(u)}(\,\cdot\,)$ converge to the finite dimensional distributions of $\widetilde{M}(\,\cdot\,)$ as $u \to 0 $.

Since $\widetilde{M}(\,\cdot\,)$ is constructed from a Poisson process, using its properties one can easily obtain that $\widetilde{M}(\,\cdot\,)$ is a.s. continuous at $t=0$ and $t=1$. In order to obtain $M_{1}$ convergence of  $M^{(u)}(\,\cdot\,)$ to $\widetilde{M}(\,\cdot\,)$ as $u \to 0$, according to Proposition~\ref{pr:Skorohod}, we need only to show (\ref{e:oscM1}), i.e
$$\lim_{\delta \to 0} \limsup_{u \to 0} \Pr ( \omega_{\delta}(M^{(u)}(\,\cdot\,)) > \epsilon ) =0.$$
Note that since $M^{(u)}(\,\cdot\,)$ is increasing, for $t_{1} \leq t \leq t_{2}$ it holds that $M^{(u)}(t_{1}) \leq M^{(u)}(t) \leq M^{(u)}(t_{2})$, which implies $M(M^{(u)}(t_{1}), M^{(u)}(t), M^{(u)}(t_{2})) =0$. Hence $\omega_{\delta}(M^{(u)})=0$, and (\ref{e:oscM1}) holds. Therefore $M^{(u)}(\,\cdot\,) \dto \widetilde{M}(\,\cdot\,)$ in $D[0,1]$ with the $M_{1}$ topology.

So far we obtained $M_{n}^{(u)}(\,\cdot\,) \dto M^{(u)}(\,\cdot\,)$ as $n \to \infty$, and $M^{(u)}(\,\cdot\,) \dto \widetilde{M}(\,\cdot\,)$ as $u \to 0$. If we show
$$ \lim_{u \to 0}\limsup_{n \to \infty} \Pr(d_{M_{1}}(M_{n}(\,\cdot\,),M_{n}^{(u)}(\,\cdot\,)) > \epsilon)=0,$$
for every $\epsilon >0$, then by Theorem 3.5 in Resnick~\cite{Resnick07} we will have, as $n \to \infty$,
$$ M_{n}(\,\cdot\,) \dto \widetilde{M}(\,\cdot\,)$$
in $D[0,1]$ with the $M_{1}$ topology. Take an arbitrary (and
fixed) $\epsilon >0$. Using the fact that the Skorohod $M_{1}$
metric on $D[0,1]$ is bounded above by the uniform metric on
$D[0,1]$ and relation (\ref{e:maxineq}) we obtain
\begin{eqnarray*}
\Pr(d_{M_{1}}(M_{n}(\,\cdot\,),M_{n}^{(u)}(\,\cdot\,)) > \epsilon)
& &\\[0.4em]
& \hspace*{-14em} \leq & \hspace*{-6.4em} \Pr \Big( \sup_{t \in
[0,1]}|M_{n}(t) - M_{n}^{(u)}(t)|
>\epsilon \Big)\\[0.4em]
& \hspace*{-14em} = & \hspace*{-6.4em} \Pr \bigg( \sup_{t \in
[0,1]} \bigg| \bigvee_{j=1}^{\lfloor nt
\rfloor}\frac{X_{j}}{a_{n}} - \bigvee_{j=1}^{\lfloor nt
\rfloor}\frac{X_{j}}{a_{n}} 1_{\big\{ \frac{X_{j}}{a_{n}} >u
\big\}} \bigg| >
\epsilon \Bigg )\\[0.4em]
& \hspace*{-14em} \leq & \hspace*{-6.4em} \Pr \bigg( \sup_{t \in
[0,1]} \bigg| \bigvee_{j=1}^{\lfloor nt
\rfloor}\frac{X_{j}}{a_{n}} 1_{\big\{ \frac{X_{j}}{a_{n}} \leq u
\big\}} \bigg| >
\epsilon \Bigg )\\[0.4em]
& \hspace*{-14em} \leq & \hspace*{-6.4em} \Pr \bigg(
\bigvee_{j=1}^{n}\frac{X_{j}}{a_{n}} 1_{\big\{
\frac{X_{j}}{a_{n}} \leq u \big\}} >
\epsilon \Bigg )\\[0.4em]
\end{eqnarray*}
Note that the last term above is equal to zero for $u \in
(0,\epsilon)$. Hence
$$ \lim_{u \to 0}\limsup_{n \to \infty} \Pr(d_{M_{1}}(M_{n}(\,\cdot\,),M_{n}^{(u)}(\,\cdot\,)) > \epsilon)=0,$$
and this concludes the proof.
\end{proof}

\begin{rem}\label{r:j1m1}
The $M_{1}$ convergence in Theorem~\ref{t:functconvergence} in
general can not be replaced by the $J_{1}$ convergence. This is
shown in Example~\ref{ex:MMp}.

The problem in our proof if we consider the $J_{1}$ topology is Lemma~\ref{l:contfunct}, which in this case does not hold.
To see this, fix $u>0$ and define
  $$ \eta_{n} = \delta_{(\frac{1}{2}- \frac{1}{n}, 2u)} + \delta_{(\frac{1}{2}, 3u)} \qquad \textrm{for} \  n \geq 3.$$
 Then $\eta_{n} \vto \eta$, where
  $$\eta = \delta_{(\frac{1}{2}, 2u)} + \delta_{(\frac{1}{2}, 3u)}.$$
  For $t_{n} = \frac{1}{2} - \frac{1}{n}$ and every strictly increasing continuous function $\lambda \colon [0,1] \to [0,1]$ such that $\lambda(0)=0$ and $\lambda(1)=1$, we have
  $$ \phi^{(u)}(\eta_{n})(t_{n}) = 2u \qquad \textrm{and} \qquad \phi^{(u)}(\eta) (\lambda(t_{n})) \in \{0, 3u\}.$$
  Therefore for every $n \geq 3$,
  $$ \| \phi^{(u)}(\eta_{n}) - \phi^{(u)}(\eta \circ \lambda)\|_{[0,1]} \geq | \phi^{(u)}(\eta_{n})(t_{n}) - \phi^{(u)}(\eta) (\lambda(t_{n}))| \geq u,$$
  and by the definition of the $J_{1}$ metric $d_{J_{1}}$ (see for example Resnick~\cite{Resnick87}, Section 4.4.1) we obtain
  $$ d_{J_{1}}( \phi^{(u)}(\eta_{n}), \phi^{(u)}(\eta \circ \lambda)) \geq u,$$
  which means that $\phi^{(u)}(\eta_{n})$ does not converge to $\phi^{(u)}(\eta)$ in the $J_{1}$ topology, i.e. the maximum functional $\phi^{(u)}$ is not continuous at $\eta$ with respect to the Skorohod $J_{1}$ topology. Since $\eta \in \Lambda$ we conclude that $\phi^{(u)}$ is not continuous on the set $\Lambda$.

 In our case the $J_{1}$ topology is inappropriate as the partial maxima process may exhibit rapid successions of
jumps within temporal clusters of large values, collapsing in the limit to a single jump. In other words the $J_{1}$ convergence could hold only if extreme values do not cluster. Since our conditions do not prohibit clustering of extremes, the $J_{1}$ convergence fails to hold, and hence the weaker $M_{1}$ topology has to be used.

\end{rem}

\begin{rem}
Theorems \ref{t:notimeconv} and \ref{t:functconvergence} can be extended to real-valued random variables, in the sense that convergence in distribution of $a_{n}^{-1}M_{n}$ and $M_{n}(\,\cdot\,)$ can be derived analogously with the use of absolute values of the variables $X_{i}$, $Z_{ij}$ and $Y_{i}$ in appropriate places. But with the methods used for positive random variables, one can not obtain a explicit form for the distribution of the limits, i.e. the distribution function of $M$ and the exponent measure of $\widetilde{M}(\,\cdot\,)$.
\end{rem}

\section{Examples}
\label{S:Examples}

We give three examples of time series that satisfy all conditions in Theorems~\ref{t:notimeconv} and~\ref{t:functconvergence}, namely joint regular variation property and Conditions~\ref{c:anticluster} and \ref{c:mixcond}. Hence for these processes we obtain convergence of partial maxima $M_{n}$ and functional convergence of partial maxima processes $M_{n}(\,\cdot\,)$. We also identify the distribution of the corresponding limits $M$ and $\widetilde{M}(\,\cdot\,)$ by indicating explicitly the extremal index $\theta$. Recall $\Pr(M \leq x) = e^{-\theta x^{-\alpha}}$ and $\Pr(\widetilde{M}(t) \leq x) = e^{-t\theta x^{-\alpha}}$ for $x>0$ and $ t>0$.

\begin{exmp}\label{ex:MMp} (Moving maxima)
Consider the finite order moving maxima defined by
$$ X_{n} = \max_{i=0,\ldots,m}\{c_{i}Z_{n-i}\}, \qquad n \in
\mathbb{Z},$$
 where $m \in
 \mathbb{N}$, $c_{0}, \ldots, c_{m}$ are nonnegative constants
 such that at least $c_{0}$ and $c_{m}$ are not equal to $0$ and $Z_{i}$, $i \in \mathbb{Z}$, are i.i.d.
  unit Fr\'{e}chet random variables, i.e. $\Pr(Z_{i} \leq x) = e^{-1/x}$ for $x>0$. Hence $Z_{i}$ is regularly varying with index $\alpha=1$. Take a sequence of positive real numbers $(a_{n})$ such that
\begin{equation*}
  n \Pr (Z_{1}>a_{n}) \to 1 \qquad \textrm{as} \ n \to \infty.
\end{equation*}
Then every $X_{i}$ is also regularly varying with index $\alpha=1$.
 Assume also (without loss of generality) that $\sum_{i=0}^{m} c_{i} = 1$. Then
 $ n \Pr (X_{1}>a_{n}) \to 1$ as $n \to \infty$.
 Since the sequence $(X_{n})$ is
 $m$--dependent, it is also strongly mixing, and therefore the
 mixing
 Condition~\ref{c:mixcond} holds. By the same property, for $s >m$ we have
 \begin{eqnarray*}
    \Pr \biggl( \max_{s \leq |i| \leq r_{n}} X_{i} > ua_{n}\,\bigg|\,X_{0}>ua_{n} \biggr) &=& \Pr \biggl( \max_{s \leq |i| \leq r_{n}} X_{i} > ua_{n} \biggr) \\[0.5em]
    & \leq & \frac{2r_{n}}{n} \cdot n \Pr(X_{1}>ua_{n}).
 \end{eqnarray*}
 Note that the expression on the right hand side in the above inequality converges to $0$ as $n \to \infty$, and hence Condition~\ref{c:anticluster} also holds. By an application of Theorem 2.3 in Meinguet~\cite{Me12} we obtain that the sequence $(X_{n})$ is jointly regularly varying with index $\alpha=1$. The extremal index of the sequence $(X_{n})$ is given by $\theta = \max_{0 \leq i \leq m} \{c_{i}\}$ (see Ancona-Navarrete and Tawn~\cite{ANT00} and Meinguet~\cite{Me12}).

 In the rest of the example we show that the $M_{1}$ convergence in Theorem~\ref{t:functconvergence} in general can not be replaced by the $J_{1}$ convergence. We use, with appropriate modifications, the procedure of Avram and Taqqu~\cite{AvTa92} in the proof of their Theorem 1. For simplicity take $m=2$. Then we have $X_{n}= \max\{ c_{0}Z_{n}, c_{1}Z_{n-1}\}$ and
 $$ M_{n}(t) = \bigvee_{i=1}^{\lfloor nt
   \rfloor}\frac{X_{i}}{a_{n}}, \qquad t \in [0,1].$$
 By Remark~\ref{r:J1conv} it suffices to prove
 \begin{equation}\label{e:osc1}
 \lim_{\delta \to 0} \limsup_{n \to \infty} \Pr ( \omega_{\delta}'(M_{n}(\,\cdot\,)) > \epsilon ) > 0
 \end{equation}
 for some $\epsilon >0$.

 Assume $c_{1} > c_{0}$. Let $i'=i'(n)$ be the index at which $\max_{1 \leq i \leq n-1}Z_{i}$ is obtained. Fix $\epsilon >0$ and introduce the events
 $$A_{n,\epsilon} = \{ Z_{i'} > \epsilon a_{n} \} = \Big\{ \max_{1 \leq i \leq n-1}Z_{i} > \epsilon a_{n}\Big\}$$
 and
 $$ B_{n,\epsilon} = \{Z_{i'}>\epsilon a_{n} \ \textrm{and} \ \exists\,l
 \neq 0, -i' \leq l \leq 1, \ \textrm{such that} \ Z_{i'+l} > \lambda \epsilon a_{n} \},$$
 where $\lambda = c_{0} / (2c_{1})$. Using the facts that $(Z_{i})$ is an i.i.d. sequence and $n
 \Pr(Z_{1}>\epsilon a_{n}) \to 1/\epsilon$ as $n \to \infty$
 (which follows from the regular variation property of $Z_{1}$) we obtain
 \begin{equation}\label{e:limAn}
  \lim_{n \to \infty}\Pr(A_{n,\epsilon}) = \lim_{n \to \infty} [1- (1-\Pr(Z_{1}> \epsilon
 a_{n})^{n-1}] = 1 - e^{-1/\epsilon}. \end{equation}
 Observe that
 $$ B_{n,\epsilon} \subseteq \bigcup_{i=1}^{n-1} \bigcup_{\footnotesize \begin{array}{c}
                                l=-(n-1) \\
                                l \neq 0
                              \end{array}}^{1}
 \{ Z_{i} > \epsilon a_{n}, Z_{i+l} > \lambda \epsilon a_{n} \}.$$
 Then it holds that
 \begin{equation}\label{e:limBn} \Pr(B_{n,\epsilon})  \leq  (n-1)n \Pr(Z_{1}>\epsilon a_{n}) \Pr(Z_{1}> \lambda \epsilon a_{n}) \to \frac{1}{\lambda \epsilon^{2}} \\
 \end{equation}
as $n \to \infty$.

On the event $A_{n,\epsilon} \setminus B_{n,\epsilon}$ one has $Z_{i'+l} \leq \lambda \epsilon a_{n}$ for every $l \neq 0$, $-i' \leq l \leq 1$, so that
$$ \bigvee_{j=1}^{i'-1}\frac{X_{j}}{a_{n}} = \max\bigg\{ c_{0} \bigvee_{j=1}^{i'-1}\frac{Z_{j}}{a_{n}}, c_{1} \bigvee_{j=0}^{i'-2}\frac{Z_{j}}{a_{n}} \bigg\} \leq c_{1} \lambda \epsilon = \frac{c_{0} \epsilon}{2}$$
and
$$ \bigvee_{j=1}^{i'}\frac{X_{j}}{a_{n}} = \max\bigg\{ \bigvee_{j=1}^{i'-1}\frac{X_{j}}{a_{n}}, \frac{X_{i'}}{a_{n}} \bigg\} \geq \frac{X_{i'}}{a_{n}} = \max \bigg\{ c_{0} \frac{Z_{i'}}{a_{n}}, c_{1} \frac{Z_{i'-1}}{a_{n}} \bigg\} \geq c_{0} \frac{Z_{i'}}{a_{n}} \geq c_{0} \epsilon.$$
Therefore
\begin{equation}\label{e:inc1}
\Big| M_{n} \Big( \frac{i'}{n} \Big) - M_{n} \Big( \frac{i'-1}{n} \Big) \Big| = \bigg| \bigvee_{j=1}^{i'}\frac{X_{j}}{a_{n}} - \bigvee_{j=1}^{i'-1}\frac{X_{j}}{a_{n}} \bigg| \geq c_{0}\epsilon - \frac{c_{0}\epsilon}{2} = \frac{c_{0}\epsilon}{2}.
\end{equation}
 On the event  $A_{n,\epsilon} \setminus B_{n,\epsilon}$ one also have
 $$  \bigvee_{j=1}^{i'}\frac{X_{j}}{a_{n}} = \frac{X_{i'}}{a_{n}} = c_{0}\frac{Z_{i'}}{a_{n}}$$
 and
 $$ \bigvee_{j=1}^{i'+1}\frac{X_{j}}{a_{n}} \geq \frac{X_{i'+1}}{a_{n}} = \max\bigg\{ c_{0} \frac{Z_{i'+1}}{a_{n}}, c_{1} \frac{Z_{i'}}{a_{n}} \Bigg\} \geq c_{1}\frac{Z_{i'}}{a_{n}},$$
 which imply
 \begin{eqnarray}\label{e:inc2}
  \nonumber \Big| M_{n} \Big( \frac{i'+1}{n} \Big) - M_{n} \Big( \frac{i'}{n} \Big) \Big| & = & \bigg| \bigvee_{j=1}^{i'+1}\frac{X_{j}}{a_{n}} - \bigvee_{j=1}^{i'}\frac{X_{j}}{a_{n}} \bigg|\\[0.4em]
  & \geq & (c_{1}-c_{0}) \frac{Z_{i'}}{a_{n}} = (c_{1}-c_{0}) \epsilon.
 \end{eqnarray}
 From (\ref{e:inc1}) and (\ref{e:inc2}) we obtain
 \begin{eqnarray*}\label{e:inc3}
 \nonumber \omega_{2/n}'(M_{n}(\,\cdot\,)) & \geq & \min \Big\{ \Big| M_{n} \Big(\frac{i'}{n} \Big) - M_{n} \Big( \frac{i'-1}{n} \Big) \Big|, \Big| M_{n} \Big(\frac{i'+1}{n} \Big) - M_{n} \Big( \frac{i'}{n} \Big) \Big| \Big\}\\[0.4em]
 & \geq & \epsilon\,\min \Big\{ \frac{c_{0}}{2}, c_{1}-c_{0} \Big\} >0
 \end{eqnarray*}
 on the event $A_{n,\epsilon} \setminus B_{n,\epsilon}$. Therefore, since $\omega_{\delta}'(\,\cdot\,)$ is nondecreasing in $\delta$, it holds that
 \begin{eqnarray*}
  \liminf_{n \to \infty} \Pr(A_{n,\epsilon} \setminus B_{n,\epsilon}) & \leq & \liminf_{n \to \infty}
 \Pr (\omega_{2/n}' (M_{n}(\,\cdot\,)) \geq  \epsilon\,\min \{ c_{0}/2, c_{1}-c_{0} \})\\[0.4em]
 & \leq &   \lim_{\delta \to 0} \limsup_{n \to \infty}  \Pr (\omega_{\delta}' (M_{n}(\,\cdot\,)) \geq  \epsilon\,\min \{ c_{0}/2, c_{1}-c_{0} \}).
 \end{eqnarray*}
   Hence if we prove $\liminf_{n \to \infty} \Pr(A_{n,\epsilon} \setminus B_{n,\epsilon}) >0$ for some $\epsilon >0$, then (\ref{e:osc1}) will hold, and this will exclude the $J_{1}$ convergence.
Since $x^{2}(1-e^{-1/x})$ tends to infinity as $x \to \infty$, we can find $\epsilon >0$ such that  $\epsilon^{2}(1-e^{-1/ \epsilon}) > 1/ \lambda$, i.e.
 $$ 1-e^{-1/ \epsilon} > \frac{1}{\lambda \epsilon^{2}}.$$
 Hence, taking into account relations (\ref{e:limAn}) and (\ref{e:limBn}) we obtain
 $$\lim_{n \to \infty} \Pr(A_{n,\epsilon}) > \limsup_{n \to \infty} \Pr(B_{n,\epsilon}),$$
 and from this immediately follows
 $$  \liminf_{n \to \infty} \Pr(A_{n,\epsilon} \setminus B_{n,\epsilon}) \geq \lim_{n \to \infty}\Pr(A_{n,\epsilon}) - \limsup_{n \to \infty} \Pr(B_{n,\epsilon}) >0.$$
 Therefore the $J_{1}$ convergence does not hold.
\end{exmp}

\smallskip

\begin{exmp} (squared GARCH process)
 We consider a stationary squared GARCH(1,1) process $(X_{n}^{2})$, where
 $$X_{n}=\sigma_{n} Z_{n},$$
with $(Z_{n})$ being a sequence of i.i.d.\ random variables with $\E (Z_{1}) = 0$ and $\operatorname{var}(Z_{1}) = 1$, and
\begin{equation}\label{e:stochvol}
    \sigma_{n}^{2} = \alpha_{0} + (\alpha_{1} Z_{n-1}^{2} +
    \beta_{1}) \sigma_{n-1}^{2},
\end{equation}
with positive parameters $\alpha_{1}, \beta_{1}$ and $\alpha_{0}$. Assume that
\begin{equation*}
-\infty \leq \E \ln (\alpha_{1}Z_{1}^{2} + \beta_{1}) < 0.
\end{equation*}
Then there exists a strictly stationary solution to the stochastic recurrence equation (\ref{e:stochvol}); see Goldie~\cite{Goldie91} and Mikosch and St\u{a}ric\u{a}~\cite{MiSt00}. The process
$(X_{n})$ is then strictly stationary too.

Assume that $Z_{1}$ is symmetric, has a positive Lebesgue density on $\mathbb{R}$ and there exists $\alpha >0$ such that
\begin{equation*}
  \E [(\alpha_{1}Z_{1}^{2} + \beta_{1})^{\alpha}]=1
  \quad \textrm{and} \quad
  \E [(\alpha_{1}Z_{1}^{2} + \beta_{1})^{\alpha} \ln (\alpha_{1}Z_{1}^{2} + \beta_{1})] < \infty.
\end{equation*}
Then it is known that the processes $(\sigma_{n}^{2})$ and
$(X_{n}^{2})$ are jointly regularly varying with index $\alpha$ and
strongly mixing with geometric rate (see Basrak et al.~\cite{BDM02b}; Mikosch and St\u{a}ric\u{a}~\cite{MiSt00}). Therefore
the sequence $(X_{n}^{2})$ satisfies Condition \ref{c:mixcond}.
Condition~\ref{c:anticluster} for the sequence $(X_{n}^{2})$
follows immediately from the results in Basrak et al.~\cite{BDM02b}. The extremal index of the sequence $(X^{2}_{n})$ is given by
$$\theta = \lim_{k \to \infty} \E \bigg( |Z_{1}|^{2\alpha} - \max_{j=2,\ldots,k+1} \Big| Z_{j}^{2} \prod_{i=1}^{j}(\alpha_{1}Z_{i-1}^{2}+\beta_{1}) \Big|^{\alpha} \bigg)_{+}\,\Big/ \E |Z_{1}|^{2 \alpha}$$
 (see Mikosch and St\u{a}ric\u{a}~\cite{MiSt00}).
\end{exmp}

\begin{exmp} (ARMAX process)
The ARMAX process is defined by
\begin{equation}\label{e:armax}
 X_{n} = \max\{cX_{n-1}, Z_{n}\}, \qquad n \in \mathbb{Z},
\end{equation}
where $0 < c < 1$ and $(Z_{n})$ is a sequence of i.i.d. random variables with
  unit Fr\'{e}chet distribution. According to Proposition 2.2 in Davis and Resnick~\cite{DaRe89} the unique stationary solution to (\ref{e:armax}) is given by
  $$X_{n} = \bigvee_{i=0}^{\infty}c^{i}Z_{n-i}.$$
  The process $(X_{n})$ is strongly mixing (see for example Ferreira and Ferreira~\cite{FeFe12}, Proposition 3.1) and therefore Condition~\ref{c:mixcond} holds. The joint regular variation property and Condition~\ref{c:anticluster} for the process $(X_{n})$ can be obtained by an application of Theorem 2.3 and Theorem 2.4 in Meinguet~\cite{Me12}. The extremal index of the sequence $(X_{n})$ is given by $\theta = 1-c$ (see Ancona-Navarrete and Tawn~\cite{ANT00} and Ferreira and Ferreira~\cite{FeFe12}).
\end{exmp}


\begin{thebibliography}{00}

\bibitem{Ad78}
R. J. Adler, Weak convergence results for extremal processes generated by dependent random variables, {\it Ann. Probab.} {\bf 6} (1978), 660--667.

\bibitem{ANT00}
M. A. Ancona-Navarrete and J. A. Tawn, A Comparison of Methods for Estimating the Extremal Index, {\it Extremes} {\bf 3} (2000) 5--38.

\bibitem{AvTa92}
F. Avram and M. Taqqu, Weak convergence of sums of moving
averages in the $\alpha$--stable domain of attraction, {\it Ann.
Probab.} {\bf 20} (1992), 483--503.

\bibitem{BDM02b}
B. Basrak, R. A. Davis and T. Mikosch, Regular variation
of GARCH processes, {\it Stoch. Process. Appl.} {\bf 99} (2002)
95--115.

\bibitem{BKS}
B. Basrak, D. Krizmani\'{c} and J. Segers, A functional limit theorem for partial sums
of dependent random variables with infinite variance, {\it Ann. Probab.} {\bf 40} (2012), 2008--2033.

\bibitem{BaSe}
B. Basrak and J. Segers, Regularly varying multivariate time
series, {\it Stoch. Process. Appl.} {\bf 119} (2009), 1055--1080.

\bibitem{DH95}
R. A. Davis and T. Hsing, Point process and partial sum
convergence for weakly dependent random variables with infinite
variance, {\it Ann. Probab.} {\bf 23} (1995), 879--917.

\bibitem{DaMi98}
R. A. Davis and T. Mikosch, The sample autocorrelations of
heavy-tailed processes with applications to ARCH, {\it Ann.
Statist.} {\bf 26} (1998), 2049--2080.

\bibitem{DaRe85}
R. A. Davis and S. I. Resnick, Limit theory for moving averages of random variables with regularly varying tail probabilities, {\it Ann. Probab.} {\bf 13} (1985), 179--195.

\bibitem{DaRe89}
R. A. Davis and S. I. Resnick, Basic properties and prediction of max-ARMA processes, {\it Adv. Appl. Probab.} {\bf 21} (1989), 781--803.

\bibitem{FeFe12}
M. Ferreira and H. Ferreira, Extremes of multivariate ARMAX processes, {\it TEST} {\bf 22} (2013), 606--627.

\bibitem{Goldie91}
C. M. Goldie, Implicit renewal theory and tails of solutions
of random equations, {\it Ann. Appl. Probab.} {\bf 1} (1991), 126--166.

\bibitem{HHL88}
T. Hsing, J. H\"{u}sler and M. R. Leadbetter, On the exceedance point process for a stationary sequence, {\it Probab. Theory Related Fields} {\bf 78} (1988), 97--112.

\bibitem{Jo09}
P. K. Jordanova, Maxima of moving averages with noise in the Weibull max-domain of attraction, {\it Math. and Education in Math.} {\bf 38} (2009), 184--189.

\bibitem{Kr10}
D. Krizmani\'c, {\it Functional limit theorems for weakly
dependent regularly varying time series}, Ph.D. thesis.
http://www.math.uniri.hr/$\sim$dkrizmanic/DKthesis.pdf (2010). Accessed 10 May 2013.

\bibitem{La64}
J. Lamperti, On extreme order statistics, {\it Ann. Math. Statist.} {\bf 35} (1964), 1726--1737.

\bibitem{Le74}
M. R. Leadbetter, On extreme values in stationary sequences, {\it Z. Wahrscheinlichkeitstheorie und Verw. Gebiete} {\bf 28} (1974), 289--303.

\bibitem{Le76}
M. R. Leadbetter, Weak convergence of high level exceedances by a stationary sequence, {\it Z. Wahrscheinlichkeitstheorie und Verw. Gebiete} {\bf 34} (1976), 11--15.

\bibitem{MeSt09}
M. M. Meerschaert and S. A. Stoev, Extremal limit theorems for observations separated by random power law waiting times, {\it J. Statist. Plann. Inference} {\bf 139} (2009), 2175--2188.

\bibitem{Me12}
T. Meinguet, Maxima of moving maxima of continuous functions, {\it Extremes} {\bf 15} (2012), 267--297.

\bibitem{MiSt00}
T. Mikosch and C. St\u{a}ric\u{a}, Limit theory for the
sample autocorrelations and extremes of a GARCH(1,1) process.
{\it Ann. Statist.} {\bf 28} (2000), 1427--1451.

\bibitem{Resnick86}
S. I. Resnick, Point processes, regular variation and weak
convergence, {\it Adv. Appl. Probab.} {\bf 18} (1986), 66--138.

\bibitem{Resnick87}
S. I. Resnick, {\it Point Processes, Regular Variation and Point Processes}, Springer-Verlag, New York, 1987.

\bibitem{Resnick07}
S. I. Resnick, {\it Heavy-Tail Phenomena: Probabilistic nad
Statistical Modeling}, Springer Science+Business Media LLC, New
York, 2007.

\bibitem{SiTe98}
D. S. Silvestrov and J. L. Teugels, Limit theorems for extremes with random simple size, {\it Adv. in Appl. Probab.} {\bf 30} (1998), 777--806.

\bibitem{Sk56}
A. V. Skorohod, Limit theorems for stochastic processes,
{\it Theor. Probab. Appl.} {\bf 1} (1956), 261--290.

\bibitem{We75}
I. Weissman, On weak convergence of extremal processes, {\it Ann. Probab.} {\bf 4} (1975), 470--473.

\bibitem{Whitt02}
W. Whitt, {\it Stochastic-Process Limits}, Springer-Verlag
LLC, New York, 2002.


\end{thebibliography}

\end{document}